\documentclass[11pt,reqno]{amsart}
\usepackage{amscd,graphicx,psfrag,amsxtra,amsmath,xcolor}
\usepackage{hyperref}

\makeatletter
\@namedef{subjclassname@2020}{%
  \textup{2020} Mathematics Subject Classification}
\makeatother

\newcommand{\p}{\partial}
\newcommand{\Z}{\mathbb Z}

\renewcommand{\H}{\mathcal H}

\newcommand{\D}{\mathcal D}
\newcommand{\Sm}{\mathcal S}

\newcommand{\C}{\mathbb C}

\newcommand{\R}{\mathbb R}

\renewcommand{\phi}{\varphi}

\renewcommand{\Re}{\operatorname{Re}}
\renewcommand{\Im}{\operatorname{Im}}

\newcommand{\spinc}{\ifmmode{\operatorname{spin}^c}\else{$\operatorname{spin}^c$\ }\fi}

\newcommand{\zp}{Z_{\infty}}

\newtheorem{theorem}{Theorem}[section]

\newtheorem{lemma}[theorem]{Lemma}
\newtheorem{proposition}[theorem]{Proposition}
\newtheorem{corollary}[theorem]{Corollary}
\theoremstyle{definition}

\newtheorem{remark}[theorem]{Remark}

\title{On the spectral sets of Inoue surfaces}

\author[Daniel Ruberman]{Daniel Ruberman}
\address{Department of Mathematics, MS 050\newline\indent Brandeis
University \newline\indent Waltham, MA 02454}
\email{\rm{ruberman@brandeis.edu}}
\author[Nikolai Saveliev]{Nikolai Saveliev}
\address{Department of Mathematics\newline\indent
University of Miami, PO Box 249085
\newline\indent Coral Gables, FL 33124}
\email{\rm{saveliev@math.miami.edu}}

\thanks{The first author was partially supported by NSF Grants DMS-1811111 and DMS-1952790, and the second author was partially supported by NSF Grant DMS-1952762}

\begin{document}
\begin{abstract}
We study the Inoue surfaces $S_M$ with the Tricerri metric and the canonical spin$^c$ structure, and the corresponding chiral Dirac operators twisted by a flat $\mathbb C^*$--connection. The twisting connection is determined by $z \in \mathbb C^*$, and the points for which the twisted Dirac operators $\mathcal D^{\pm}_z$ are not invertible are called spectral points. We show that there are no spectral points inside the annulus $\alpha^{-1/4} < |z| < \alpha^{1/4}$, where $\alpha >1$ is the only real eigenvalue of the matrix $M$ that determines $S_M$, and find the spectral points on its boundary. Via Taubes' theory of end-periodic operators, this implies that the corresponding Dirac operators are Fredholm on any end-periodic manifold whose end is modeled on $S_M$.
\end{abstract}

\maketitle
\section{Introduction}
Inoue surfaces are compact complex surfaces with zero second Betti number which are most remarkable in that they contain no holomorphic curves. These surfaces, constructed by Inoue \cite{inoue}, belong to the class $\rm{VII}_{\,0}$ in Kodaira's classification~\cite{barth-hulek-peters-vandeven}, which is to say that they are minimal connected compact complex surfaces $X$ with Kodaira dimension $\kappa(X) = -\infty$ and the first Betti number $b_1 (X) = 1$. In fact, any class $\rm{VII}_{\,0}$ surface with vanishing second Betti number and no holomorphic curves is biholomorphic to an Inoue surface; see Bogomolov \cite{Bogomolov1, Bogomolov2} and Teleman \cite{Teleman}. Inoue surfaces, which are not K{\"a}hler because their first Betti number is odd, have been extensively studied from the viewpoints of both algebraic and differential geometry. 

In this paper, we restrict ourselves to the Inoue surfaces $X$ of class $S_M$ associated with certain integral matrices $M \in SL(3,\Z)$ with one real eigenvalue $\alpha > 1$ and two complex eigenvalues $\beta \ne \bar\beta$. These surfaces, described in detail in Section \ref{S:def}, are known to be diffeomorphic to the mapping torus of a self-diffeomorphism of the 3-torus induced by $M$. It is in this incarnation that the surfaces $S_M$ are best known to topologists. In particular, Cappell and Shaneson \cite{cappell-shaneson,cappell-shaneson:knots} independently used some of the matrices $M$ to construct a fake $\R{\rm P}^4$ and interesting fibered $2$-spheres in a homotopy $4$-sphere.  From this point of view, the manifolds $S_M$ are given by surgery on this homotopy 4-sphere along those knots. The question of when this homotopy 4-sphere is in fact diffeomorphic to $S^4$ has received considerable attention~\cite{akbulut-kirby:cs1,aitchison-rubinstein:involutions,akbulut-kirby:cs2,gompf:sphere}.

Inoue surfaces are an intriguing class of examples to which to apply our work on the Seiberg--Witten invariants \cite{MRS1} and the end-periodic index theorem \cite{MRS2}. Spectral properties of chiral Dirac operators $\D^{\pm}(X)$ play an important role in determining the index of associated Dirac operators on end-periodic manifolds whose end is modeled on an infinite cyclic cover of $X$. In applications of those papers to date~\cite{lin-ruberman-saveliev:double,lin-ruberman-saveliev:splitting,lin-ruberman-saveliev:finite} the infinite cyclic cover was a Riemannian product of the real line and a $3$-manifold. 
 In the case of an Inoue surface, while this cover is topologically the product of the real line and a $3$-torus, it is not a metric product. (This is related to the fact that the monodromy of the bundle $X \to S^1$ has infinite order.)  Since the end-periodic index is metric dependent, this makes for an index problem that must be investigated analytically. We study this problem for the Tricerri metric on $X$, which makes it into a locally conformal K{\"a}hler manifold, and the canonical $\spinc$ structure; see Section \ref{S:def}. 

More specifically, we are interested in the spectral sets of the associated chiral Dirac operators $\D^{\pm}(X)$. Recall from \cite{MRS1} that $z \in \mathbb C^*$ is a  spectral point of $\D^{\pm}(X)$ if and only if the operator 
\[
z^{f}\, \circ\, \D^{\pm} (X)\, \circ\, z^{-f} = \D^{\pm} (X) - \ln z\cdot df
\]
has non-zero kernel, where $f: X \to S^1$ is a smooth function realizing a generator of $H^1(X; \Z) = \Z$, and $df$ operates by Clifford multiplication. One can easily check that the spectral sets of $\D^+(X)$ and $\D^-(X)$ are obtained from each other by inversion $\tau(z) = 1/\bar z$ with respect to the unit circle. The following theorem, which was announced in~\cite[Section 6.4]{MRS2}, is the main result of this paper. 

\begin{theorem}\label{T:unit}
The operators $\D^{\pm} (X)$ have no spectral points in the annulus $\alpha^{-1/4} < |z| < \alpha^{1/4}$. Furthermore, the only spectral points of $\D^+ (X)$ on the circles $|z| = \alpha^{-1/4}$ and $|z| = \alpha^{1/4}$ are, respectively, $z= \alpha^{1/4}\beta$ and $z = \alpha^{1/4}$. 
\end{theorem}

Let $\zp$ be a $\spinc$ end-periodic manifold whose end is modeled on the infinite cyclic cover of an Inoue surface $X$. According to Taubes \cite[Lemma 4.3]{T}, the Dirac operators $\D^{\pm}(\zp)$ are Fredholm in the usual Sobolev $L^2$ completion if and only if their spectral sets are disjoint from the unit circle $|z| = 1$. 

\begin{corollary}\label{C:unit}
The operators $\D^{\pm}(\zp): L^2_1 (\zp) \to L^2 (\zp)$ are Fredholm on any end-periodic $\spinc$ manifold $\zp$ whose end is modeled on an Inoue surface $X$ of type $S_M$. 
\end{corollary}

\begin{remark} Inoue surfaces do not admit metrics of positive scalar metric, as was proved by Albanese \cite[Theorem 4.5]{A}. This also follows from Cecchini and Schick \cite{CS}, making use of the fact that Inoue surfaces are solvmanifolds (see Wall~\cite{wall:geometries,wall:geometric} and Hasegawa~\cite{H}) and hence are enlargeable in the sense of Gromov and Lawson \cite{GL}. In particular, one cannot prove that the operators $\D^{\pm}(\zp)$ of Corollary \ref{C:unit} are Fredholm by using the (uniformly) positive scalar curvature at infinity condition as in~\cite{GL}. 
\end{remark}

Once we establish that the operators $\D^{\pm} (\zp)$ are Fredholm, their index can in principle be calculated as in \cite{MRS2} in terms of an integral term and the periodic eta-invariant $\eta (X)$. The latter is a spectral invariant which generalizes the eta-invariant of Atiyah, Patodi, and Singer \cite{APS} and which can be viewed as a regularized count of points in the spectral set of $\D^{\pm}(X)$. The partial information about the spectral set we obtain in this paper is not sufficient to calculate $\eta(X)$ or the associated index of $\D^{\pm} (\zp)$. However, even this modest attempt leads to some fascinating analysis which we felt was worth sharing.

It is worth mentioning that our original interest in end-periodic index theory grew out of our work \cite{MRS1} with Mrowka on Seiberg--Witten theory for $4$-manifolds $X$ with $b_2 (X) = 0$ and $b_1 (X) = 1$. In that paper, a Seiberg--Witten invariant $\lambda_{\rm{SW}}(X)$ was defined as a sum of two metric dependent terms. One is a count of solutions to the Seiberg--Witten equations, and the other is an index-theoretic correction term, whose most important part is the index of the Dirac operator $\D^+(\zp)$. 

Evaluating $\lambda_{\rm{SW}}(X)$ for an Inoue surface $X$ presents quite a challenge. One can actually solve a modified version of the Seiberg--Witten equations for the Tricerri metric -- see~\cite{okonek-teleman:b+0,lupascu:sw-hermitian}. However, the modification involves a certain twisting of the Dirac operator used in the formulation of the Seiberg--Witten equations. In order to turn this into a calculation of $\lambda_{\rm{SW}}(X)$, one would have to first relate this modified Seiberg--Witten equation to the one used in~\cite{MRS2}. The second step would be to evaluate the correction term; this is essentially the same as finding the invariant $\eta(X)$. As mentioned above, we are quite far from achieving this.    

In conclusion, we mention that a recent paper of Holt and Zhang~\cite{holt-zhang:kt} uses related techniques to investigate $\bar\partial$--harmonic forms on a different non-K\"{a}hler complex manifold, the Kodaira--Thurston surface~\cite{kodaira:I,thurston:symplectic}.


\medskip\noindent
\textbf{Acknowledgments:}\; 
We thank Tom Mrowka, Leonid Parnovski, and Andrei Teleman for generously sharing their expertise, and the anonymous referee for useful comments.


\section{Inoue surfaces}\label{S:def}
The Inoue surfaces $X$ we are interested in are all compact quotients of $\H \times \C$, where $\H = \{\,w = w_1 + i w_2\in \C\;|\;w_2 > 0\,\}$ is the upper complex half-plane. To construct $X$, start with an integral matrix $M \in SL(3,\Z)$ with one real eigenvalue $\alpha > 1$ (which must therefore be irrational) and two complex conjugate eigenvalues $\beta \neq \bar\beta$. For example, the matrices
\[
A_m = 
\begin{pmatrix}
\; 0 & 1 & 0 \\
\; 0 & 1 & 1 \\
\; 1 & 0 & m + 1 
\end{pmatrix}
,
\]

\smallskip\noindent
which are equivalent to the Cappell and Shaneson \cite{cappell-shaneson:knots} family, will do as long as $-2 \le m \le 3$. Let $a = (a_1,a_2,a_3)$ be a real eigenvector corresponding to $\alpha$, and $b = (b_1,b_2,b_3)$ a complex eigenvector corresponding to $\beta$. Let $G_M$ be the group of complex analytic transformations of $\H \times \C$ generated by
\begin{gather}
g_0 (w,z) = (\alpha w, \beta z), \notag \\
g_i (w,z) = (w + a_i,z + b_i),\quad i = 1, 2, 3. \notag
\end{gather}

\smallskip\noindent
The group $G_M$ acts on $\H \times \C$ freely and properly discontinuously so that the quotient $X = (\H \times \C)/G_M$ is a compact complex surface. 

Inoue \cite{inoue} showed that as a smooth manifold $X$ is a 3-torus bundle over a circle whose monodromy is given by the matrix $M$, and that $b_1 (X) = 1$ and $b_2 (X) = 0$. One can check, for example, that $H_*(X) = H_*(S^1 \times S^3)$ for all manifolds $X$ obtained from the Cappell--Shaneson matrices $A_m$. Define a function $f: \H \times \C \to \R$ by the formula $f(w,z) = \ln w_2/\ln \alpha$. One can easily see that $df$ is a well defined 1-form on $X$, whose cohomology class generates $H^1 (X; \Z) = \Z$. 

The complex surface $X$ admits no global K{\"a}hler metric. We will however consider the following Hermitian metric on $\H \times \C$, called the Tricerri metric,
\[
g\; =\; \frac {dw\otimes d\bar w}{w_2^2}\; +\; w_2\;dz\otimes d\bar z,
\]
see \cite{Tricerri, DO}. Let $\omega$ be the K{\"a}hler form associated with this metric then $d\omega = d\ln w_2 \wedge \omega$, with the torsion form $d\ln w_2 = \ln \alpha\cdot df$. The metric $g$ is $G_M$--invariant, hence it defines a metric on $X$ which makes $X$ into a locally conformal K{\"a}hler manifold (or l.c.K. manifold, for short). 

The complex surface $X$ admits a canonical $\spinc$ structure with respect to which 
\[
\Sm^+ = \Lambda^{0,0} (X) \,\oplus\,\Lambda^{0,2}(X)\quad\text{and}\quad \Sm^- = \Lambda^{0,1} (X).
\] 
Let $\D^{\pm}(X)$ be the chiral Dirac operators associated with the Tricerri metric and the canonical $\spinc$ structure on $X$. These are the operators that Theorem \ref{T:unit} is concerned with. The proof of Theorem \ref{T:unit} will take up the rest of these notes. 


\section{Reduction to the Dirac--Dolbeault operator}
Let $\D^-(X)$ be the negative chiral Dirac operator associated with the Tricerri metric and the canonical $\spinc$ structure on $X$. According to Gauduchon \cite[page 283]{G}, there is an isomorphism 
\begin{equation}\label{E:gaud}
\D^- (X) + \frac 1 4\, \ln\alpha\cdot df\;=\; \sqrt{\,2}\,(\bar\p\,\oplus\,\bar\p^*),
\end{equation}
where
\begin{equation}\label{E:dd}
\bar\p\,\oplus\,\bar\p^*: \;\Omega^{0,1} (X) \longrightarrow \Omega^{0,2}(X)\,\oplus\,\Omega^{0,0}(X)
\end{equation}
is the Dirac--Dolbeault operator on the complex surface $X$. To prove Theorem \ref{T:unit}, it will suffice to compute the spectral set of \eqref{E:dd}. The spectral set of $\D^-(X)$ will be obtained from it via multiplication by $\alpha^{-1/4}$, and the spectral set of $\D^+(X)$ by further inversion.


\section{The periodic boundary value problem}  To compute the spectral set of \eqref{E:dd}, we will complete the operator \eqref{E:dd} to an operator $L^2_1 \to L^2$ and look for $z = e^{\mu}\in \C^*$ such that the kernel of the operator
\[
e^{\mu f} \circ ( \bar\p\,\oplus\,\bar\p^*) \circ e^{-\mu f}\; =\; ( \bar\p\,\oplus\,\bar\p^*) - \mu\cdot df
\]
on $X$ is non-zero. Equivalently, after passing to the universal covering space $\H \times \C \to X$, we will look for $\mu$ such that the following periodic boundary problem on $\H \times \C$ has a non-zero solution $\omega \in \Omega^{0,1} (\H \times \C)$:
\begin{gather}
(\bar\p\,\oplus\,\bar\p^*)(\omega) = 0,\quad\text{where}\notag \\
g_i^*\,\omega = \omega\;\;\text{for}\;\; i = 1, 2, 3,\quad\text{and}\quad
g_0^*\,\omega = e^{-\mu}\cdot\omega. \notag
\end{gather}

\noindent
Let us re-state this periodic boundary problem by writing $\omega = a\,d\bar w + b\,d\bar z$ on $\H\times \C$. The equation $(\bar\p\,\oplus\,\bar\p^*)(\omega) = 0$ turns into the system

\[
\begin{cases}
\;\;\dfrac{\p a}{\p\bar z}\; -\; \dfrac {\p b}{\p\bar w} = 0 \\
\vspace{-4mm} \\
\;\;\dfrac{\p (w_2\,a)}{\p w}\;+\;\dfrac 1 {w_2^2}\cdot\dfrac {\p b}{\p z} = 0
\end{cases}
\]

\medskip\noindent
and, after introducing the new function $c = w_2\,a$ and the new variable $t = \ln w_2$, into the system
\begin{equation}\label{E:one}
\left(\frac {\p}{\p t} + i\,B_t\right)\,\begin{pmatrix} b \\ c \end{pmatrix} 
= 0
\end{equation}
with
\[
B_t = \begin{pmatrix}
-e^t\,\dfrac {\p}{\p w_1} & 2\,\dfrac{\p}{\p\bar z} \\
\vspace{-4mm} \\
2e^{-t}\,\dfrac{\p}{\p z} & e^t\,\dfrac{\p}{\p w_1}
\end{pmatrix}.\quad
\]

\bigskip
Taking into account the periodic boundary conditions $g_i^*\,\omega = \omega\;\;\text{for}\;\; i = 1, 2, 3$, this can be viewed as a system on the product $\R \times T^3$, with the coordinates $t$ on the real line and $(w_1,z_1,z_2)$ on the torus $T^3$. The remaining periodic boundary condition $g_0^*\,\omega = e^{-\mu}\cdot\omega$ can be expressed in the language of $(0,1)$--forms as 
\[
g_0^*\,(a(w,z)\,d\bar w + b(w,z)\,d\bar z) = e^{-\mu}\cdot(a(w,z)\,d\bar w + 
b(w,z)\,d\bar z).
\]
After switching to $c = w_2\cdot a$, this turns into 
\begin{equation}\label{E:boundary}
\bar\beta\cdot b(\alpha w,\beta z) = e^{-\mu}\cdot b(w,z)\quad\text{and}\quad 
c(\alpha w,\beta z) = e^{-\mu}\cdot c(w,z).
\end{equation}

\noindent
It is the periodic boundary value problem \eqref{E:one}, \eqref{E:boundary} on the manifold $\mathbb R \times T^3$ that we now wish to solve.


\section{Fourier analysis}
We will use  Fourier analysis on the 3-torus to solve the system \eqref{E:one}. First, consider the following basis in $\mathbb R^3$\,:
\begin{gather*}
\xi   = (a_1,\; \Re b_1,\; \Im b_1) \\
\eta  = (a_2,\; \Re b_2,\; \Im b_2) \\ 
\zeta = (a_3,\; \Re b_3,\; \Im b_3)
\end{gather*}
where $a = (a_1, a_2, a_3)$ and $b = (b_1, b_2, b_3)$ are, as before, the eigenvectors of $M$ corresponding to the eigenvalues $\alpha$ and $\beta$. The quotient of $\mathbb R^3$ by the integer lattice spanned by the vectors $\xi$, $\eta$, $\zeta$ is our 3-torus. The matrix whose rows are the vectors $\xi$, $\eta$, $\zeta$ will be called $Y$ so that
\[
Y =
\begin{pmatrix}
\xi_1 &\; \xi_2 &\; \xi_3 \\
\eta_1 &\; \eta_2 &\; \eta_3 \\
\zeta_1 &\; \zeta_2 &\; \zeta_3
\end{pmatrix}
.
\quad\empty
\]

\smallskip\noindent
Without loss of generality, we will assume that $\det Y = 1$. The columns of the matrix 
\[
Y^{-1} =
\begin{pmatrix}
\xi_1^* &\; \eta_1^* &\; \zeta_1^* \\
\xi_2^* &\; \eta_2^* &\; \zeta_2^* \\
\xi_3^* &\; \eta_3^* &\; \zeta_3^*
\end{pmatrix}
\qquad\empty
\]

\smallskip\noindent
form the dual basis $\xi^*$, $\eta^*$, $\zeta^*$ with respect to the usual dot product $(\,\cdot\,,\cdot)$ on $\mathbb R^3$. One can easily check that the functions $T^3 \to \C$ defined by
\begin{equation}\label{E:basis}
\theta\; \to\; \exp\,(2\pi i\,(\theta,k\xi^* + \ell\eta^* + m\zeta^*))\quad\text{for all}
\quad (k,\ell,m) \in \mathbb Z^3,
\end{equation}
where $\theta = (\theta_1,\theta_2,\theta_3) = (w_1,z_1,z_2)$, form an 
orthonormal basis in the $L^2$-space of complex-valued functions on the 
3-torus. 

For each $t\in \mathbb R$, expand the functions $b(t,\theta)$ and  $c(t,\theta): T^3 \to \C$ into Fourier series,
\[
b (t,\theta) = \sum_{k,\ell,m}\;b_{k \ell m} (t)\,\exp\,(2\pi i\,(\theta,k\xi^* 
+ \ell\eta^* + m\zeta^*))
\]
and 
\[
c (t,\theta) = \sum_{k,\ell,m}\;c_{k \ell m} (t)\,\exp\,(2\pi i\,(\theta,k\xi^* 
+ \ell\eta^* + m\zeta^*)),
\]
and plug them into equation \eqref{E:one}. For each individual triple of integers $(k,\ell,m)$, we obtain the system
\begin{equation}\label{E:two}
\begin{pmatrix} b'_{k \ell m} \\ c'_{k \ell m}\end{pmatrix} = 
\begin{pmatrix} -e^t P_{k \ell m} &  Q_{k \ell m} \\ e^{-t} \bar Q_{k \ell m} & 
e^t P_{k \ell m}\end{pmatrix}
\begin{pmatrix} b_{k \ell m} \\ c_{k \ell m} \end{pmatrix},
\end{equation}

\medskip\noindent
where the prime stands for the $t$-derivative,
\[
P_{k \ell m} = 2\pi (k\xi_1^* + \ell\eta_1^* + m\zeta_1^*) \in \mathbb R,\quad\text{and}
\]
\[
Q_{k \ell m} = 2\pi (k\xi_2^* + \ell\eta_2^* + m\zeta_2^*) + 2\pi i\, 
(k\xi_3^* + \ell\eta_3^* + m\zeta_3^*) \in \mathbb C.
\]
This is a linear system of ordinary differential equations with non-constant coefficients. Note that $P_{k \ell m}$ and $Q_{k \ell m}$ are actually constants so the only dependence of the coefficients on $t$ comes from the factors of $e^t$ and $e^{-t}$. For future use, we make the following observation.

\begin{lemma}\label{L:galois}
For no choice of $(k,\ell,m) \neq (0,0,0)$ can $Q_{k\ell m}$ be equal to zero.
\end{lemma}

\begin{proof}
Observe that
\[
Y\,
\begin{pmatrix} P_{k \ell m} \\ \Re\, Q_{k \ell m} \\ \Im\, Q_{k \ell m} 
\end{pmatrix}\; =\; 2\pi\,\begin{pmatrix} k \\ \ell \\ m \end{pmatrix}.
\]
If $Q_{k\ell m} = 0$, the first column of $Y$, which is an eigenvector of $M$ with the eigenvalue $\alpha$, is proportional to the vector with integral coordinates $k$, $\ell$, and $m$. The latter vector is then also an eigenvector of $M \in SL(3,\Z)$ with the eigenvalue $\alpha$, which contradicts the fact that $\alpha$ is irrational.
\end{proof}

Next, we need to take care of the boundary conditions \eqref{E:boundary}. In 
our $\theta$--notations, we have $\beta z = (\beta_1 + i\beta_2)(z_1 + 
i z_2) = (\beta_1 + i\beta_2)(\theta_2 + i \theta_3) = (\beta_1\theta_2 
- \beta_2\theta_3) + i (\beta_2\theta_2 + \beta_1\theta_3)$ and $\alpha w = 
\alpha (w_1 + iw_2) = \alpha\theta_1 + i e^{t + \ln \alpha}$. To simplify
notations, introduce the matrix 
\[
A = \begin{pmatrix} \alpha & 0 & 0 \\
0 & \beta_1 & -\beta_2 \\
0 & \beta_2 &  \beta_1 
\end{pmatrix}
\]
then the boundary conditions \eqref{E:boundary} become
\[
\bar\beta\cdot b(t + \ln \alpha,A(\theta)) = e^{-\mu}\cdot b(t,\theta),
\quad
c(t + \ln \alpha,A(\theta)) = e^{-\mu}\cdot c(t,\theta).
\]
In order to re-write these in terms of the Fourier coefficients $b_{k \ell m}$ 
and $c_{k \ell m}$, we need the following technical result. 

\begin{lemma}
For any integers $k$, $\ell$ and $m$, we have $(A(\theta),k\xi^* + \ell\eta^*
+ m\zeta^*) = (\theta,k'\xi^* + \ell'\eta^* + m'\zeta^*)$, where
\begin{equation}\label{E:M}
\begin{pmatrix} k' \\ \ell' \\ m' \end{pmatrix}\; =\; M\,
\begin{pmatrix} k \\ \ell \\ m \end{pmatrix}.
\end{equation}
\end{lemma}

\begin{proof}
A straightforward calculation with matrices shows that $M Y = Y A^t$. Viewing $\theta$ as a column, we obtain
\begin{multline}\notag
(A(\theta),k\xi^* + \ell\eta^* + m\zeta^*) \\ = 
\theta^t A^t\, Y^{-1} \begin{pmatrix} k \\ \ell \\ m \end{pmatrix} = 
\theta^t\, Y^{-1} M \begin{pmatrix} k \\ \ell \\ m \end{pmatrix}  =
\theta^t\, Y^{-1} \begin{pmatrix} k' \\ \ell' \\ m' \end{pmatrix} \\ =
(\theta,k'\xi^* + \ell'\eta^* + m'\zeta^*).
\end{multline}
\end{proof}

Now, substitute the Fourier expansions of $b(t,\theta)$ and $c(t,\theta)$ into the boundary conditions to obtain
\begin{alignat*}{1}
\bar\beta \cdot b(t + \ln \alpha,&\, A(\theta)) \\ 
& =\bar\beta\; \sum_{k,\ell, m} b_{\,k \ell m} (t + \ln \alpha)\,\exp(2\pi i (A(\theta),
k\xi^* + \ell\eta^* + m\zeta^*)) \\ 
& =\bar\beta\; \sum_{k,\ell, m} b_{\,k \ell m} (t + \ln \alpha)\,\exp(2\pi i (\theta,
k'\xi^* + \ell'\eta^* + m'\zeta^*)) \\ 
& = e^{-\mu}\,\sum_{k',\ell', m'} b_{k' \ell' m'} (t)\,
\exp(2\pi i (\theta,k'\xi^* + \ell'\eta^* + m'\zeta^*)),
\end{alignat*}
and similarly for $c$. A term-by-term comparison of the coefficients allows us to conclude that
\begin{equation}\label{E:bdry2}
\bar\beta \cdot b_{k\ell m}(t + \ln \alpha) = e^{-\mu}\cdot b_{k'\ell'm'}(t),
\quad
c_{k\ell m}(t + \ln \alpha) = e^{-\mu}\cdot c_{k'\ell'm'}(t),
\end{equation}
where the triples $(k,\ell,m)$ and $(k',\ell',m')$ are related by the equation \eqref{E:M}. Therefore, to fit  $b_{k\ell m}(t)$ and $c_{k\ell m}(t)$ together into a Fourier series solution, we need to know how $M$ acts on the triples $(k,\ell,m)$.


\section{Finite orbits}\label{S:finite}
The infinite cyclic subgroup of $SL (3, \mathbb Z)$ generated by the matrix $M$ acts on the lattice $\mathbb Z^3$. The only finite orbit of this action consists of the triple $(k,\ell ,m) = (0,0,0)$. The solutions of the equation \eqref{E:two} corresponding to this triple must be constant; we will denote them by $b$ and $c$. The boundary conditions \eqref{E:bdry2} then translate into $\bar\beta\, b = e^{-\mu}\, b$ and $c = e^{-\mu} c$, resulting in exactly two choices for the spectral point $z = e^{\mu}$ of the operator \eqref{E:dd}, namely, $z = 1$ and $z = 1/\bar\beta = \alpha\beta$. These correspond to the spectral points $z = \alpha^{1/4}$ and $z = \alpha^{1/4}\beta$ of the operator $\D^+(X)$ as claimed in Theorem \ref{T:unit}.


\section{Infinite orbits}
For any fixed triple of integers $(k_0,\ell_0,m_0) \neq (0,0,0)$, the triples $(k_n,\ell_n,m_n)$, $n \in \Z$, in its orbit can be found from the equation

\[
\begin{pmatrix} k_n \\ \ell_n \\ m_n \end{pmatrix}\; =\; M^n\, 
\begin{pmatrix} k_0 \\ \ell_0 \\ m_0 \end{pmatrix}.
\]

\medskip\noindent
Denote $b_n (t) = b_{k_n \ell_n m_n}(t)$ and $c_n (t) = c_{k_n \ell_n m_n}(t)$. It follows from equations \eqref{E:bdry2} that, once we know $b_0 (t)$ and $c_0 (t)$, the rest of $b_n (t)$ and $c_n (t)$ can be determined uniquely from the recursive relation
\[
b_{n+1} (t) = \bar\beta\cdot e^{\mu} \cdot b_n (t + \ln \alpha),\quad
c_{n+1} (t) = e^{\mu} \cdot c_n (t + \ln \alpha).
\]
Therefore, each infinite orbit gives rise to the infinite series
\medskip
\begin{alignat*}{1}
b(t,\theta) & = \sum_{n \in \mathbb Z}\quad\bar\beta^n\cdot e^{n\mu} \cdot b_0 (t 
+ n\ln\alpha)\cdot\exp(2\pi i(\theta,k_n\xi^* + \ell_n\eta^* + m_n\zeta^*)),\\
c(t,\theta) & = \sum_{n \in \mathbb Z}\quad e^{n\mu} \cdot c_0(t + n\ln\alpha)
\cdot\exp(2\pi i(\theta,k_n\xi^* + \ell_n\eta^* + m_n\zeta^*)).
\end{alignat*}

\smallskip\noindent
The question becomes whether these formal series solutions converge to a solution of \eqref{E:one}. We will show that, for certain values of $\mu$, the series cannot converge in $L^2$ norm unless $b_0 (t) = c_0 (t) = 0$; this  will imply that the corresponding $z = e^{\mu}$ are not in the spectral set of the operator $\bar\p\,\oplus\,\bar\p^*$. To this end, denote by $\delta$ the real number 
\[
\delta\; =\; \Re \mu/\ln \alpha - 1/4
\]
and introduce the notations
\[
u(t) = b_0 (t)\quad\text{and}\quad v(t) = e^{\,t/2}\,c_0(t).
\]

\begin{lemma}\label{L:xy}
The above Fourier series for $b(t,\theta)$ and $c(t,\theta)$ converge to $L^2_1$ sections on $X$ if and only if both $u(t)$ and $v(t)$ belong to $L^2_{1,\,\delta - 1/4}\,(\mathbb R)$.
\end{lemma}

\begin{proof}
Let $z = \bar\beta \cdot e^{\mu}$ then $z^{\,t/\!\ln\alpha}\cdot b(t,\theta)$ is the Fourier--Laplace transform \cite{MRS1} of the function $u (t) \exp(2\pi i(\theta,k_0\xi^*+\ell_0\eta^*+m_0\zeta^*))$ on $\mathbb R \times T^3$ with respect to covering translation $(t,\theta) \to (t\, +\, \ln \alpha, A(\theta))$. One can easily check that
\[
\big|z^{1/\!\ln\alpha}\big| 
= e^{\,\delta - 1/4}.
\] 
From this point on, we follow the proof of \cite[Proposition 4.2]{MRS1} and use the fact that the functions $\exp(2\pi i(\theta, k_n \xi^* + \ell_n\eta^* + m_n\zeta^*))$ form an orthonormal basis on the fibers $\{ t\} \times T^3$. For example, it follows by direct calculation that
\begin{multline*}
\| z^{\,t/\!\ln \alpha}\cdot b(t,\theta) \|_{L^2(X)}^2\, =\, 
\sum_{n\in \mathbb Z}\;\; \int_0^{\ln \alpha} |z|^{\,2(n+ t/\!\ln \alpha)}\cdot | u(t + n \ln\alpha)|^2\,dt \\ =
\int_{-\infty}^{\infty} |z|^{\,2t/\!\ln\alpha}\cdot |u(t)|^2\,dt \, =\, \| u \|^2_{L^2_{\delta - 1/4} (\mathbb R)}\,.
\end{multline*}
The proof for the function $c(t,\theta)$ is similar.
\end{proof}



One can easily check using \eqref{E:two} that the functions $u(t)$ and $v(t)$ solve the following system of ordinary differential equations 

\begin{equation}\label{E:three}
\begin{pmatrix} u' \\ v'\end{pmatrix} =
\begin{pmatrix} -P e^t   &  Q\,e^{-t/2} \\ 
\bar Q\,e^{-t/2}   &   1/2 + P e^t \end{pmatrix}\,
\begin{pmatrix} u \\ v\end{pmatrix},
\end{equation}

\medskip\noindent
where $P = P_{k_0 \ell_0 m_0} \in \R$ and $Q = Q_{k_0 \ell_0 m_0} \in \C$. Because of Lemma \ref{L:xy}, we are only interested in solutions $u(t)$ and $v(t)$ which belong to $L^2_{1,\,\delta - 1/4}\,(\mathbb R)$.

\begin{proposition}\label{P:spec}
Suppose that $-1/4 \le \delta \le 1/4$ then all solutions $u(t)$, $v(t)$ of the system \eqref{E:three} which belong to $L^2_{1,\,\delta - 1/4}\,(\mathbb R)$ are identically zero. 
\end{proposition}

\begin{proof}
De-coupling equations \eqref{E:three} turns them into the following Sturm--Liou\-ville problems 
\begin{gather}
-u'' + (Pe^t (Pe^t - 1) + |Q|^2\,e^{-t})\,u = 0\quad\text{and}\label{E:four} \\
-v'' + (Pe^t (Pe^t + 2) + |Q|^2\,e^{-t} + 1/4)\,v = 0.\label{E:five}
\end{gather}
Without loss of generality, we will assume that $u$ and $v$ are real valued functions. We will separate our argument into three cases, depending on whether $P$ is positive, negative, or zero.

If $P < 0$, introduce the positive real numbers $p = -P$ and $q = |Q|$ and re-write the equation \eqref{E:four} in the form $-u'' + U(t)\,u = 0$ with the everywhere positive potential $U(t) = pe^t (pe^t + 1) + q^2\,e^{-t}$. For any choice of $a < b$, we then have
\[
- \int_a^b u''(t)u(t)\,dt + \int_a^b U(t)\,u^2(t)\,dt = 0
\]

\medskip\noindent
and, after integration by parts,
\medskip
\begin{equation}\label{E:ab}
\int_a^b u'(t)^2\,dt + u(a)u'(a) - u(b)u'(b)  + \int_a^b U(t)\,u(t)^2\,dt = 0.
\end{equation}

\smallskip\noindent
The first and the last terms in this formula are non-negative for any choice of $a < b$. We will show that there exist $a$ arbitrarily close to $-\infty$ and $b$ arbitrarily close to $+\infty$ such that the other two terms in \eqref{E:ab} are non-negative as well. This will imply that $u(t) = 0$. Plugging $u(t) = 0$ back into \eqref{E:three} will then imply that $v(t) = 0$ because $Q \ne 0$ by Lemma \ref{L:galois}. 

We first show that for any $a_0$ there exists $a \le a_0$ such that $u(a)u'(a) \ge 0$. If $u(a_0) = 0$ we are finished. Otherwise, suppose that $u(t)u'(t) < 0$ for all $t \le a_0$. Then $(u^2(t))' = 2u(t)u'(t) < 0$ so that $u^2(t)$ is a decreasing function and hence $u^2(t) \ge u^2(a_0) > 0$ for all $t \le a_0$. This contradicts the fact that $u \in L^2_{\delta-1/4}(\mathbb R)$ with $\delta \le 1/4$.

Next, we show that for any $b_0$ there exists $b \ge b_0$ such that $u(b)u'(b) \le 0$. If $u(b_0) = 0$ we are finished. Otherwise, suppose that $u(t)u'(t) > 0$ for all $t \ge b_0$. Then $(u^2(t))' = 2u(t)u'(t) > 0$ so that $u^2(t)$ is an increasing function and hence $u^2(t) \ge u^2(b_0) > 0$ for all $t \ge b_0$. Using the formula \eqref{E:ab} with $a = b_0$ we obtain the estimate
\[
u(b) u'(b)\;\ge\;\int_{b_0} ^b U(t)\,u^2(t)\,dt\;\ge\;u^2(b_0)\int_{b_0}^b U(t)\,dt,
\]

\smallskip\noindent
and using the fact that $U(t) \ge p^2\,e^{2t}$ for all $t$, the estimate
\smallskip
\[
u(b) u'(b)\;\ge\;\frac 1 2\;p^2\,u^2(b_0)\left(e^{2b} - e^{2b_0}\right)\quad\text{for all}\quad b \ge b_0.
\]

\smallskip\noindent
Since $u(t)$ and $u'(t)$ belong to $L^2_{\delta - 1/4} (\mathbb R)$, it follows from the H{\"o}lder inequality that $u(t) u'(t) \in L^1_{2(\delta - 1/4)} (\mathbb R)$. This contradicts the above estimate for $\delta \ge -1/4$.

If $P > 0$, essentially the same argument using equation \eqref{E:five} shows that $v(t) = 0$. After plugging $v(t) = 0$ back in \eqref{E:three}, we see that $u(t) = 0$ as well. 

In the remaining case of $P = 0$, both equations \eqref{E:four} and \eqref{E:five} admit explicit solutions in terms of Bessel functions. To be precise, the general solution of \eqref{E:five} is of the form 
\begin{equation}\label{E:bessel}
C_1\cdot I_1 (2q e^{-t/2}) + C_2\cdot K_1 (2q e^{-t/2}),
\end{equation}
where $I_1(x)$ and $K_1(x)$ are the modified Bessel functions of the first and second kind, solving the equation $x^2 y'' + x y' - (x^2 + 1)\,y = 0$. One can check that the zero function is the only function among \eqref{E:bessel} that belongs to $L^2_{\delta - 1/4} (\mathbb R)$ with $-1/4 \le \delta \le 1/4$.
\end{proof}

Proposition \ref{P:spec} together with the discussion in Section \ref{S:finite} completes the proof of Theorem \ref{T:unit}.


\end{document}